\documentclass[preprint,12pt]{elsarticle1}

\usepackage{subfig}
\usepackage{graphicx}
\usepackage{caption,color}   
\usepackage{amssymb}
\usepackage{amsthm}
\usepackage{amsmath}
\usepackage{epic}
\usepackage{setspace}
\usepackage{float}
\usepackage{multirow}

\usepackage{hyperref}
\usepackage{xcolor}
\usepackage{marginnote}

\newtheorem{thm}{Theorem}[section]
\newtheorem{cor}[thm]{Corollary}
\newtheorem{con}[thm]{Conjecture}
\newtheorem{lem}[thm]{Lemma}

\numberwithin{equation}{section}
\journal{}

\begin{document}
\begin{spacing}{1.15}
\begin{frontmatter}
\title{\textbf{The multiplicity of the zero Laplacian eigenvalue of uniform hypertrees}}

\author{Ge Lin}\ead{linge0717@126.com}
\author{Changjiang Bu}\ead{buchangjiang@hrbeu.edu.cn}

\address{College of Mathematical Sciences, Harbin Engineering University, Harbin, PR China}

\begin{abstract}
In this paper, the Laplacian characteristic polynomial of uniform hypergraphs with cut vertices or pendant edges
and the Laplacian matching polynomial of uniform hypergraphs are characterized.
The multiplicity of the zero Laplacian eigenvalue of uniform hypertrees is given, which proves the conjecture in \cite{zheng2023zero} (The zero eigenvalue of the Laplacian tensor of a uniform hypergraph, Linear and Multilinear Algebra, (2023) Doi:10.1080/03081087.2023.2172541).
\end{abstract}


\end{frontmatter}

\section{Introduction}

A hypergraph is called $k$-uniform if its each edge contains exactly $k$ vertices.
For a $k$-uniform hypergraph $H=(V(H),E(H))$,
its adjacency tensor $\mathcal{A}_{H}=(a_{i_{1}i_{2}\cdots i_{k}})$ is a $k$-order $|V(H)|$-dimensional tensor \cite{cooper2012spectra},
where
\begin{equation*}
a_{i_{1}i_{2}\cdots i_{k}}=\begin{cases}
\frac{1}{(k-1)!}&\text{if $\{i_{1},i_{2},\ldots,i_{k}\}\in E(H)$},\\
0&\text{otherwise}.
\end{cases}
\end{equation*}
The Laplacian tensor of $H$ is $\mathcal{L}_{H}=\mathcal{D}_{H}-\mathcal{A}_{H}$ \cite{qi2014h},
where $\mathcal{D}_{H}$ is the diagonal tensor of vertex degrees of $H$.
The eigenvalues of $\mathcal{A}_{H}$ and $\mathcal{L}_{H}$ are called the eigenvalues and Laplacian eigenvalues of $H$, respectively.
The characteristic polynomial of $\mathcal{A}_{H}$ and $\mathcal{L}_{H}$ are called the characteristic polynomial and the Laplacian characteristic polynomial of $H$, respectively.

The characteristic polynomials of uniform hypergraphs are a research area that has attached much attention in spectral hypergraph theory.
In 2012, Cooper and Dutle \cite{cooper2012spectra} characterized some properties on the characteristic polynomials of uniform hypergraphs
and gave the characteristic polynomial of the one-edge hypergraph.
In 2015, Cooper and Dutle \cite{cooper2015computing} gave the characteristic polynomial of the $3$-uniform hyperstar.
In 2020, Bao et al. \cite{bao2020combinatorial} provided a combinatorial method for computing the characteristic polynomial of uniform hypergraphs with cut vertices, and gave the characteristic polynomial of the $k$-uniform hyperstar.
In 2021, Chen and Bu \cite{chen2021reduction} gave a reduction formula for the characteristic polynomial of uniform hypergraphs with pendant edges.
Besides, they used the reduction formula to derive the characteristic polynomial of the uniform hyperpath.

However, there are few results on the Laplacian characteristic polynomials of uniform hypergraphs.
In 2023, Zheng \cite{zheng2023zero} gave the Laplacian characteristic polynomial of uniform hyperstar,
and obtained the multiplicity of the zero Laplacian eigenvalue of uniform hyperstar and hyperpath.
Moreover, the following conjecture was proposed in \cite{zheng2023zero}.

\begin{con}\cite{zheng2023zero}\label{caixiang1.1}
Let $T=(V(T),E(T))$ be a $k$-uniform hypertree for $k\geq3$.
Then the multiplicity of the zero Laplacian eigenvalue of $T$ is $k^{|E(T)|(k-2)}$.
\end{con}

The eigenvalues of uniform hypertrees can be studied by the matching polynomial.
In 2017, Zhang et al. \cite{zhang2017spectra} showed that the roots of the matching polynomial of a uniform hypertree are its eigenvalues.
For a $k$-uniform hypertree $T$ with $k\geq3$,
Clark and Cooper \cite{clark2018hypertrees} determined all eigenvalues of $T$ by roots of the matching polynomials of all
sub-hypertrees of $T$.
In 2022, Wan et al. \cite{wan2022spectra} defined the Laplacian matching polynomial of uniform hypergraphs,
and used the roots of the Laplacian matching polynomials of all sub-hypertrees of $T$ to obtain all Laplacian eigenvalues of $T$ (without multiplicity).

In this paper, we give a expression for the Laplacian characteristic polynomial of uniform hypergraphs with cut vertices or pendant edges (Section \ref{section2}).
And we characterize some properties on the Laplacian matching polynomial of uniform hypergraphs (Section \ref{section3}).
Further, we use these results to give the multiplicity of the zero Laplacian eigenvalue of uniform hypertrees,
which shows that Conjecture \ref{caixiang1.1} is true (Section \ref{section4}).

\section{The Laplacian characteristic polynomial of uniform hypergraphs}\label{section2}

\subsection{Preliminaries}

In this subsection, we present some notation and lemmas about the eigenvalue of tensors and the formula of resultants.

A $k$-order $n$-dimensional tensor $\mathcal{A}=(a_{i_{1}i_{2}\cdots i_{k}})$ refers to a multi-dimensional array with entries $a_{i_{1}i_{2}\cdots i_{k}}$ for all $i_{j}\in[n]:=\{1,\ldots,n\}$ and $j\in[k]$.
If there exists $\lambda\in\mathbb{C}$ and a non-zero vector $\mathbf{x}=(x_{1},\ldots,x_{n})^{\mathrm{T}}\in\mathbb{C}^{n}$ such that
\begin{align*}
\mathcal{A}\mathbf{x}^{k-1}=\lambda\mathbf{x}^{[k-1]},
\end{align*}
where $\mathcal{A}\mathbf{x}^{k-1}$ is an $n$-dimensional vector with $\sum_{i_{2},\ldots,i_{k}=1}^{n}a_{ii_{2}\ldots i_{k}}x_{i_{2}}\cdots x_{i_{k}}$ as its $i$-th component
 and $\mathbf{x}^{[k-1]}=(x_{1}^{k-1},\ldots,x_{n}^{k-1})^{\mathrm{T}}$,
then $\lambda$ is called an eigenvalue of $\mathcal{A}$ and $\mathbf{x}$ is an eigenvector of $\mathcal{A}$ corresponding to $\lambda$ (see \cite{lim2005singular,qi2005eigenvalues}).
The resultant of the polynomials system $(\lambda\mathbf{x}^{[k-1]}-\mathcal{A}\mathbf{x}^{k-1})$ is called the characteristic polynomial of $\mathcal{A}$, denoted by $\phi(\mathcal{A})$.

In the following, we introduce some formulas of resultants required for proofs in this section.

\begin{lem}\citep[Poisson Formula for resultants]{gelfand1994discriminants}\label{yinli2.1}
Let $F_{1},F_{2},\ldots,F_{n}\in\mathbb{C}[x_{1},\ldots,x_{n}]$ be homogeneous polynomials of respective degrees $d_{1},d_{2},\ldots,d_{n}$.
For each $i\in[n]$,
let $\overline{F}_{i}=F_{i}|_{x_{1}=0}$ and $f_{i}=F_{i}|_{x_{1}=1}$.
Let $\mathcal{V}$ be the affine variety defined by the polynomials $f_{2},\ldots,f_{n}$.
If $\mathrm{Res}(\overline{F}_{2},\ldots,\overline{F}_{n})\neq0$, then
\begin{align*}
\mathrm{Res}(F_{1},F_{2},\ldots,F_{n})
=\mathrm{Res}(\overline{F}_{2},\ldots,\overline{F}_{n})^{d_{1}}\prod_{\mathbf{p}\in\mathcal{V}}f_{1}(\mathbf{p})^{m(\mathbf{p})},
\end{align*}
where $m(\mathbf{p})$ is the multiplicity of a point $\mathbf{p}$ in $\mathcal{V}$.
\end{lem}

\begin{lem}\citep[lemma 3.2]{cooper2012spectra}\label{yinli2.2}
Let $F_{1},\ldots,F_{n}\in\mathbb{C}[x_{1},\ldots,x_{n}]$ be homogeneous polynomials of respective degrees $d_{1},\ldots,d_{n}$,
and let $G_{1},\ldots,G_{m}\in\mathbb{C}[y_{1},\ldots,y_{m}]$ be homogeneous polynomials of respective degrees $\delta_{1},\ldots,\delta_{m}$.
Then
\begin{align*}
\mathrm{Res}(F_{1},\ldots,F_{n},G_{1},\ldots,G_{m})=\mathrm{Res}(F_{1},\ldots,F_{n})^{\prod_{j=1}^{m}\delta_{j}}\mathrm{Res}(G_{1},\ldots,G_{m})^{\prod_{i=1}^{n}d_{i}}.
\end{align*}
\end{lem}

Let $H=(V(H),E(H))$ be a $k$-uniform hypergraph with $V(H)=[n]$.
For a vertex $v\in V(H)$,
let $E_{H}(v)$ denote the set of edges of $H$ containing $v$ and $d_{H}(v)$ denote the degree of $v$ in $H$.
Given an edge $e\in E(H)$ and a vector $\mathbf{x}=(x_{1},\ldots,x_{n})^{\mathrm{T}}\in \mathbb{C}^{n}$,
let $\mathbf{x}_{e}=\prod_{v\in e}x_{v}$.
Then the eigenvalue equation $\mathcal{L}_{H}\mathbf{x}^{k-1}=\lambda\mathbf{x}^{[k-1]}$ corresponding to the Laplacian tensor of $H$ can be written as
\begin{align*}
d_{H}(v)x_{v}^{k-1}-\sum_{e\in E_{H}(v)}\mathbf{x}_{e\setminus\{v\}}=\lambda x_{v}^{k-1}, v=1,\ldots,n.
\end{align*}
For each $v\in V(H)$, define
\begin{align*}
F_{v}=(\lambda-d_{H}(v))x_{v}^{k-1}+\sum_{e\in E_{H}(v)}\mathbf{x}_{e\setminus\{v\}}.
\end{align*}
For a fixed vertex $w\in V(H)$, let
\begin{align*}
\overline{F}_{v}=F_{v}|_{x_{w}=0}, f_{v}=F_{v}|_{x_{w}=1}.
\end{align*}
Let $\mathcal{V}^{H}$ be the affine variety defined by the polynomials $f_{v}$ for all $v\in V(H)\setminus\{w\}$.
We use $\mathcal{L}_{H}(w)=(l_{i_{1}\cdots i_{k}})$ to denote a $k$-order $n-1$-dimensional principal sub-tensor of $\mathcal{L}_{H}$,
where $i_{1},\ldots,i_{k}\in V(H)\setminus\{w\}$.
By the Poisson Formula for resultants, we obtain the following lemma about the Laplacian characteristic polynomial of $H$.

\begin{lem}\label{yinli2.3}
Let $H$ be a $k$-uniform hypergraph and $w$ be a vertex on $H$.
Then the Laplacian characteristic polynomial
\begin{align}\label{shi2.1}
\phi(\mathcal{L}_{H})
=\phi(\mathcal{L}_{H}(w))^{k-1}\prod_{\mathbf{p}\in\mathcal{V}^{H}}(\lambda-d_{H}(w)+\sum_{e\in E_{H}(w)}\mathbf{p}_{e\setminus\{w\}})^{m(\mathbf{p})},
\end{align}
where $m(\mathbf{p})$ is the multiplicity of $\mathbf{p}$ in $\mathcal{V}^{H}$.
\end{lem}

\begin{proof}

By the definition of the Laplacian characteristic polynomial,
we know that $\phi(\mathcal{L}_{H})=\mathrm{Res}(F_{v}:v\in V(H))$,
where $F_{v}=(\lambda-d_{H}(v))x_{v}^{k-1}+\sum_{e\in E_{H}(v)}\mathbf{x}_{e\setminus\{v\}}$.
For the vertex $w\in V(H)$, by Lemma \ref{yinli2.1}, we have
\begin{align*}
\phi(\mathcal{L}_{H})=\mathrm{Res}(\overline{F}_{v}:v\in V(H)\setminus\{w\})^{k-1}
\prod_{\mathbf{p}\in\mathcal{V}^{H}}f_{w}(\mathbf{p})^{m(\mathbf{p})}.
\end{align*}
For all $v\in V(H)\setminus\{w\}$,
$\overline{F}_{v}=F_{v}|_{x_{w}=0}=(\lambda-d_{H}(v))x_{v}^{k-1}+\sum_{e\in E_{H-w}(v)}\mathbf{x}_{e\setminus\{v\}}=0$ are the eigenvalue equations of $\mathcal{L}_{H}(w)$,
where $H-w$ denote the hypergraph obtained from $H$ by removing the vertex $w$ and all edges incident to it,
so we have
\begin{align}\label{shi2.2}
\mathrm{Res}(\overline{F}_{v}: v\in V(H)\setminus\{w\})=\phi(\mathcal{L}_{H}(w)).
\end{align}
Note that $f_{w}=F_{w}|_{x_{w}=1}=\lambda-d_{H}(w)+\sum_{e\in E_{H}(w)}\mathbf{x}_{e\setminus\{w\}}$.
Then we obtain
\begin{align*}
\phi(\mathcal{L}_{H})
=\phi(\mathcal{L}_{H}(w))^{k-1}\prod_{\mathbf{p}\in\mathcal{V}^{H}}(\lambda-d_{H}(w)+\sum_{e\in E_{H}(w)}\mathbf{p}_{e\setminus\{w\}})^{m(\mathbf{p})}.
\end{align*}

\end{proof}

When $H$ is a uniform hypergraph with cut vertices,
we can give a description of the affine variety $\mathcal{V}^{H}$ for this case
and obtain a more explicit expression for the Laplacian characteristic polynomial of $H$ than (\ref{shi2.1}).

\subsection{Main results}

Let $H=(V(H),E(H))$ be a $k$-uniform connected hypergraph and $w\in V(H)$.
Denote $\widehat{E}_{H}(w)=\{e\setminus\{w\}:e\in E_{H}(w)\}$.
Deleting the vertex $w$,
it can get a non-uniform hypergraph $\widehat{H}$ with vertex set $V(\widehat{H})=V(H)\setminus\{w\}$ and edge set $E(\widehat{H})=(E(H)\setminus E_{H}(w))\cup \widehat{E}_{H}(w)$.
The vertex $w$ is called a cut vertex if $\widehat{H}$ is not connected \cite{bao2020combinatorial}.
Suppose that $w$ is a cut vertex on $H$ and $\widehat{H}_{1},\ldots,\widehat{H}_{s}$ are connected components of $\widehat{H}$.
For each $i\in[s]$,
denote the induced sub-hypergraph of $H$ on $V(\widehat{H}_{i})\cup\{w\}$ by $\widetilde{H}_{i}$,
and we call $\widetilde{H}_{i}$ a branch of $H$ associated with $w$.
Clearly, $H$ can be obtained by coalescing the branches $\widetilde{H}_{1},\ldots,\widetilde{H}_{s}$ to the vertex $w$.
Recall that the affine variety $\mathcal{V}^{H}$ is defined by the polynomials $f_{v}=(\lambda-d_{H}(v))x_{v}^{k-1}+\sum_{e\in E_{H}(v)}\mathbf{x}_{e\setminus\{v\}}|_{x_{w}=1}$ for all $v\in V(H)\setminus\{w\}$.
Then, for each $v_{i}\in V(\widetilde{H}_{i})\setminus\{w\}$ and $i\in[s]$, we have
\begin{align*}
f_{v_{i}}&=(\lambda-d_{H}(v_{i}))x_{v_{i}}^{k-1}+\sum_{e\in E_{H}(v_{i})}\mathbf{x}_{e\setminus\{v_{i},w\}}\\
&=(\lambda-d_{\widetilde{H}_{i}}(v_{i}))x_{v_{i}}^{k-1}+\sum_{e\in E_{\widetilde{H}_{i}}(v_{i})}\mathbf{x}_{e\setminus\{v_{i},w\}}.
\end{align*}
It is known that $\mathcal{V}^{\widetilde{H}_{i}}$ is the affine variety defined by the polynomials $f_{v_{i}}$ for all $v_{i}\in V(\widetilde{H}_{i})\setminus\{w\}$ and each $i\in[s]$.
So
\begin{align}\label{shi2.3}
\mathcal{V}^{H}=\mathcal{V}^{\widetilde{H}_{1}}\times\cdots\times\mathcal{V}^{\widetilde{H}_{s}}.
\end{align}

Combining Lemma \ref{yinli2.1} with (\ref{shi2.3}),
an expression for the Laplacian characteristic polynomial of uniform hypergraphs with cut vertices is derived as follows.

\begin{thm}\label{dingli2.4}
Let $H$ be a $k$-uniform hypergraph and $w$ be a cut vertex on $H$.
Let $\widetilde{H}_{1},\ldots,\widetilde{H}_{s}$ are the branches of $H$ associated with $w$.
Denote $\mathcal{V}^{(i)}=\mathcal{V}^{\widetilde{H}_{i}}$ and $E_{i}(w)=E_{\widetilde{H} _{i}}(w)$.
Then
\begin{align*}
\phi(\mathcal{L}_{H})=\prod_{i=1}^{s}\phi\left(\mathcal{L}_{\widetilde{H}_{i}}(w)\right)^{(k-1)^{2-s+\sum_{j\neq i}|V(\widetilde{H}_{j})|}}
\prod_{\substack{\mathbf{p}^{(i)}\in\mathcal{V}^{(i)}\\i\in[s]}}
(\lambda-\sum_{i=1}^{s}d_{\widetilde{H}_{i}}(w)+\sum_{\substack{e\in E_{i}(w)\\i\in[s]}}\mathbf{p}^{(i)}_{e\setminus\{w\}})^{\prod_{i=1}^{s}m(\mathbf{p}^{(i)})},
\end{align*}
where $m(\mathbf{p}^{(i)})$ is the multiplicity of $\mathbf{p}^{(i)}$ in $\mathcal{V}^{(i)}$ for each $i\in[s]$.
\end{thm}

\begin{proof}

By Lemma \ref{yinli2.3}, the Laplacian characteristic polynomial
\begin{align}\label{shi2.4}
\phi(\mathcal{L}_{H})=\phi(\mathcal{L}_{H}(w))^{k-1}\prod_{\mathbf{p}\in\mathcal{V}^{H}}(\lambda-d_{H}(w)+\sum_{e\in E_{H}(w)}\mathbf{p}_{e\setminus\{w\}})^{m(\mathbf{p})}.
\end{align}
From (\ref{shi2.2}), we know that $\phi(\mathcal{L}_{H}(w))=\mathrm{Res}(\overline{F}_{v}: v\in V(H)\setminus\{w\})$.
Recall that $\overline{F}_{v}=(\lambda-d_{H}(v))x_{v}^{k-1}+\sum_{e\in E_{H}(v)}\mathbf{x}_{e\setminus\{v\}}|_{x_{w}=0}$ for each $v\in V(H)\setminus\{w\}$,
and note that $H$ can be obtained by coalescing the branches $\widetilde{H}_{1},\ldots,\widetilde{H}_{s}$ to the vertex $w$.
For all $v_{i}\in V(\widetilde{H}_{i})\setminus\{w\}$ and each $i\in[s]$, we have
\begin{align*}
\overline{F}_{v_{i}}&=(\lambda-d_{H}(v_{i}))x_{v_{i}}^{k-1}+\sum_{e\in E_{H}(v_{i})}\mathbf{x}_{e\setminus\{v_{i}\}}|_{x_{w}=0}\\
&=(\lambda-d_{\widetilde{H}_{i}}(v_{i}))x_{v_{i}}^{k-1}+\sum_{e\in E_{\widetilde{H}_{i}}(v_{i})}\mathbf{x}_{e\setminus\{v_{i}\}}|_{x_{w}=0}\\
&=(\lambda-d_{\widetilde{H}_{i}}(v_{i}))x_{v_{i}}^{k-1}+\sum_{e\in E_{\widetilde{H}_{i}-w}(v_{i})}\mathbf{x}_{e\setminus\{v_{i}\}},
\end{align*}
where $\widetilde{H}_{i}-w$ denote the hypergraph obtained from $\widetilde{H}_{i}$ by removing the vertex $w$ and all edges incident to it.
So $\phi(\mathcal{L}_{H}(w))=\mathrm{Res}(\overline{F}_{v}: v\in V(H)\setminus\{w\})=\mathrm{Res}(\overline{F}_{v_{i}}: v_{i}\in V(\widetilde{H}_{i})\setminus\{w\}, i\in[s])$.
By Lemma \ref{yinli2.2}, we get
\begin{align*}
\phi(\mathcal{L}_{H}(w))=\prod_{i=1}^{s}\mathrm{Res}(\overline{F}_{v_{i}}:v_{i}\in V(\widetilde{H}_{i})\setminus\{w\})^{(k-1)^{1-s+\sum_{j\neq i}|V(\widetilde{H}_{j})|}}.
\end{align*}
For all $v_{i}\in V(\widetilde{H}_{i})\setminus\{w\}$ and each $i\in[s]$,
$\overline{F}_{v_{i}}=0$ are the eigenvalue equations of $\mathcal{L}_{\widetilde{H}_{i}}(w)$.
Then we have $\mathrm{Res}(\overline{F}_{v_{i}}: v_{i}\in V(\widetilde{H}_{i})\setminus\{w\})=\phi(\mathcal{L}_{\widetilde{H}_{i}}(w))$,
which implies that
\begin{align}\label{shi2.5}
\phi(\mathcal{L}_{H}(w))
=\prod_{i=1}^{s}\phi(\mathcal{L}_{\widetilde{H}_{i}}(w))^{(k-1)^{1-s+\sum_{j\neq i}|V(\widetilde{H}_{j})|}}.
\end{align}

For any $\mathbf{p}\in\mathcal{V}^{H}$, by (\ref{shi2.3}),
we have $\mathbf{p}=\left(\begin{matrix} \mathbf{p}^{(1)}\\ \vdots \\ \mathbf{p}^{(s)} \end{matrix}\right)$,
where $\mathbf{p}^{(i)}\in\mathcal{V}^{(i)}$ for all $i\in[s]$.
Then we obtain
\begin{align}\label{shi2.6}\notag
\prod_{\mathbf{p}\in\mathcal{V}^{H}}(\lambda-d_{H}(w)+\sum_{e\in E_{H}(w)}\mathbf{p}_{e\setminus\{w\}})^{m(\mathbf{p})}
&=\prod_{\mathbf{p}\in\mathcal{V}^{H}}(\lambda-d_{H}(w)+\sum_{\substack{e\in E_{i}(w)\\i\in[s]}}\mathbf{p}_{e\setminus\{w\}})^{m(\mathbf{p})}\\
&=\prod_{\substack{\mathbf{p}^{(i)}\in\mathcal{V}^{(i)}\\i\in[s]}}(\lambda-\sum_{i=1}^{s}d_{\widetilde{H}_{i}}(w)+\sum_{\substack{e\in E_{i}(w)\\i\in[s]}}\mathbf{p}^{(i)}_{e\setminus\{w\}})^{\prod_{i=1}^{s}m(\mathbf{p}^{(i)})}.
\end{align}

Substituting (\ref{shi2.5}) and (\ref{shi2.6}) into (\ref{shi2.4}),
the proof is completed.

\end{proof}

An edge on $k$-uniform hypergraph is called a pendant edge if it contains exactly $k-1$ vertices with degree one.
When $k$-uniform hypergraph $H$ has a pendant edge incident to $w$,
it implies that $w$ is a cut vertex on $H$ and one of the branches is the one-edge hypergraph.
We use Theorem \ref{dingli2.4} to give a more explicit expression for the Laplacian characteristic polynomial of uniform hypergraphs with pendant edges.

\begin{cor}\label{tuilun2.5}
Let $H$ be a $k$-uniform hypergraph with a pendant edge incident to the non-pendent vertex $w$,
and we define $\widetilde{H}$ as the $k$-uniform hypergraph obtained by removing the pendant edge and pendent vertices on it from $H$.
Then
\begin{align*}
\phi(\mathcal{L}_{H})=&(\lambda-1)^{(k-1)^{|V(\widetilde{H})|+k-1}}\phi(\mathcal{L}_{\widetilde{H}}(w))^{(k-1)^{k}}
\prod_{\mathbf{p}\in\mathcal{V}^{\widetilde{H}}}(\lambda-d_{\widetilde{H}}(w)-1+\sum_{e\in E_{\widetilde{H}}(w)}\mathbf{p}_{e\setminus\{w\}})^{m(\mathbf{p})K_{1}}\\
&\times\prod_{\mathbf{p}\in\mathcal{V}^{\widetilde{H}}}(\lambda-d_{\widetilde{H}}(w)-1+(\frac{-1}{\lambda-1})^{k-1}+\sum_{e\in E_{\widetilde{H}}(w)}\mathbf{p}_{e\setminus\{w\}})^{m(\mathbf{p})K_{2}},
\end{align*}
where $K_{1}=(k-1)^{k-1}-k^{k-2}$ and $K_{2}=k^{k-2}$.
\end{cor}

\begin{proof}

Clearly, $w$ is a cut vertex on $H$.
Suppose that the branches of $H$ associated with $w$ are $\widetilde{H}$ and the one-edge hypergraph with $k$ vertices, denoted by $H'$.
By Theorem \ref{dingli2.4}, we have
\begin{align}\label{shi2.7}\notag
\phi(\mathcal{L}_{H})=&\phi\left(\mathcal{L}_{\widetilde{H}}(w)\right)^{(k-1)^{k}}\phi\left(\mathcal{L}_{H'}(w)\right)^{(k-1)^{|V(\widetilde{H})|}}\\
&\times\prod_{\substack{\mathbf{p}\in\mathcal{V}^{\widetilde{H}}\\\mathbf{q}\in\mathcal{V}^{H'}}}
(\lambda-d_{\widetilde{H}}(w)-1+\mathbf{q}_{e'\setminus\{w\}}+\sum_{e\in E_{\widetilde{H}}(w)}\mathbf{p}_{e\setminus\{w\}})^{m(\mathbf{p})m(\mathbf{q})},
\end{align}
where $e'$ is the edge of $H'$.

Since $\mathcal{L}_{H'}(w)$ is a $k$-order $k-1$-dimensional identity tensor for the one-edge hypergraph $H'$,
we get
\begin{align}\label{shi2.8}
\phi(\mathcal{L}_{H'}(w))=(\lambda-1)^{(k-1)^{k-1}}.
\end{align}
It is shown that the Laplacian characteristic polynomial of $H'$ is
$\phi(\mathcal{L} _{H'})=(\lambda-1)^{k(k-1)^{k-1}-k^{k-1}}((\lambda-1)^{k}+(-1)^{k-1})^{k^{k-2}}$ in the \cite[Theorem 4.2]{zheng2023zero}.
It follows from (\ref{shi2.1}) that
\begin{align*}
\prod_{\mathbf{q}\in\mathcal{V}^{H'}}(\lambda-1+\mathbf{q}_{e'\setminus\{w\}})^{m(\mathbf{q})}&=\frac{\phi(\mathcal{L}_{H'})}{\phi\left(\mathcal{L}_{H'}(w)\right)^{k-1}}\\
&=(\lambda-1)^{(k-1)^{k-1}-k^{k-2}}(\lambda-1+(\frac{-1}{\lambda-1})^{k-1})^{k^{k-2}}.
\end{align*}
Then we have
\begin{align}\label{shi2.9}
\mathbf{q}_{e'\setminus\{w\}}=\begin{cases}
0,& \text{if $\mathbf{q}=\mathbf{0}$,}\\
(\frac{-1}{\lambda-1})^{k-1},& \text{if $\mathbf{q}\neq\mathbf{0}$,}
\end{cases}
\end{align}
for $\mathbf{q}\in\mathcal{V}^{H'}$,
and we have $m(\mathbf{0})=(k-1)^{k-1}-k^{k-2}$, $\sum_{\mathbf{0}\neq\mathbf{q}\in\mathcal{V}^{H'}}m(\mathbf{q})=k^{k-2}$ for $\mathbf{0}\in\mathcal{V}^{H'}$.
By (\ref{shi2.9}), the equation in (\ref{shi2.7}) is derived as follows:
\begin{align}\label{shi2.10}\notag
&\prod_{\substack{\mathbf{p}\in\mathcal{V}^{\widetilde{H}}\\\mathbf{q}\in\mathcal{V}^{H'}}}(\lambda-d_{\widetilde{H}}(w)-1+\mathbf{q}_{e'\setminus\{w\}}+\sum_{e\in E_{\widetilde{H}}(w)}\mathbf{p}_{e\setminus\{w\}})^{m(\mathbf{p})m(\mathbf{q})}\\ \notag
=&\prod_{\substack{\mathbf{p}\in\mathcal{V}^{\widetilde{H}}\\ \mathbf{0}=\mathbf{q}\in\mathcal{V}^{H'}}}(\lambda-d_{\widetilde{H}}(w)-1+\mathbf{q}_{e'\setminus\{w\}}+\sum_{e\in E_{\widetilde{H}}(w)}\mathbf{p}_{e\setminus\{w\}})^{m(\mathbf{p})m(\mathbf{q})}\\ \notag
&\times\prod_{\substack{\mathbf{p}\in\mathcal{V}^{\widetilde{H}}\\ \mathbf{0}\neq\mathbf{q}\in\mathcal{V}^{H'}}}(\lambda-d_{\widetilde{H}}(w)-1+\mathbf{q}_{e'\setminus\{w\}}+\sum_{e\in E_{\widetilde{H}}(w)}\mathbf{p}_{e\setminus\{w\}})^{m(\mathbf{p})m(\mathbf{q})}\\ \notag
=&\prod_{\mathbf{p}\in\mathcal{V}^{\widetilde{H}}}(\lambda-d_{\widetilde{H}}(w)-1+\sum_{e\in E_{\widetilde{H}}(w)}\mathbf{p}_{e\setminus\{w\}})^{m(\mathbf{p})((k-1)^{k-1}-k^{k-2})}\\
&\times\prod_{\mathbf{p}\in\mathcal{V}^{\widetilde{H}}}(\lambda-d_{\widetilde{H}}(w)-1+(\frac{-1}{\lambda-1})^{k-1}+\sum_{e\in E_{\widetilde{H}}(w)}\mathbf{p}_{e\setminus\{w\}})^{m(\mathbf{p})k^{k-2}}.
\end{align}

Substituting (\ref{shi2.8}) and (\ref{shi2.10}) into (\ref{shi2.7}),
the proof is completed.

\end{proof}

\section{The Laplacian matching polynomial of uniform hypergraphs}\label{section3}

Let $H=(V(H),E(H))$ be a $k$-uniform hypergraph.
Let $M$ be a sub-set of $E(H)$.
Denote by $V(M)$ the set of vertices of $H$ each of which is an endpoint of one of the edges in $M$.
If no two distinct edges in $M$ share a common vertex,
then $M$ is called a matching of $H$.
The set of matchings (including the empty set) of $H$ is denoted by $\mathcal{M}(H)$.
Let $\mathbf{w}:V(H)\cup E(H)\rightarrow \mathbb{C}$ be a weighting function on $H$.
In 2022, Wan et al. \cite{wan2022spectra} defined the weighted matching polynomial of $H$ as
\begin{align*}
\sum_{M\in\mathcal{M}(H)}(-1)^{|M|}\prod_{e\in M}\mathbf{w}(e)^{k}
\prod_{v\in V(H)\setminus V(M)}(\lambda-\mathbf{w}(v)).
\end{align*}
For any sub-hypergraph $\widetilde{H}$ of $H$,
if we choose the weighting function on $\widetilde{H}$ such that $\mathbf{w}(v) = d_{H}(v)$ for all $v\in V(\widetilde{H})$ and $\mathbf{w}(e)=-1$ for all $e\in E(\widetilde{H})$,
then the weighted matching polynomial of $\widetilde{H}$ can be derived as
\begin{align}\label{shi3.1}
\sum_{M\in\mathcal{M}(\widetilde{H})}(-1)^{(k-1)|M|}\prod_{v\in V(\widetilde{H})\setminus V(M)}(\lambda-d_{H}(v))=:\varphi_{H}(\widetilde{H}).
\end{align}
In \cite{wan2022spectra}, the polynomial (\ref{shi3.1}) is called the Laplacian matching polynomial of $\widetilde{H}$ with respect to $H$.

The goal of this section is to characterize some properties on the Laplacian matching polynomial of uniform hypergraphs,
which will be used to prove the main results in Section \ref{section4}.

Firstly, we introduce some related notation.
For a sub-set $S\subseteq V(H)$,
we use $H-S$ to denote the hypergraph obtained from $H$ by deleting the vertices in $S$ and the edges incident to them.
For a sub-set $I\subseteq E(H)$,
let $H\setminus I$ denote the hypergraph obtained from $H$ by deleting the edges in $I$ (no deletion of resultant isolated vertices).
When $S=\{v\}$ and $I=\{e\}$,
$H-S$ and $H\setminus I$ are simply written as $H-v$ and $H\setminus e$, respectively.

\begin{thm}\label{dingli3.1}
Let $H$ be a $k$-uniform hypergraph,
and $\widetilde{H}$ be a sub-hypergraph of $H$.
Then the following statements hold. \\
(1) If $\widetilde{H}$ is not connected and its connected components is $\widetilde{H}_{1}$ and $\widetilde{H}_{2}$,
then $\varphi_{H}(\widetilde{H})=\varphi_{H}(\widetilde{H}_{1})\varphi_{H}(\widetilde{H}_{2})$; \\
(2) For $e\in E(\widetilde{H})$,
we have $\varphi_{H}(\widetilde{H})=\varphi_{H}(\widetilde{H}\setminus e)+(-1)^{k-1}\varphi_{H}(\widetilde{H}-V(e))$; \\
(3) For $v\in V(\widetilde{H})$ and $I\subseteq E_{\widetilde{H}}(v)$, we have
\begin{align*}
\varphi_{H}(\widetilde{H})=\varphi_{H}(\widetilde{H}\setminus I)+(-1)^{k-1}\sum_{e\in I}\varphi_{H}(\widetilde{H}-V(e)),
\end{align*}
and
\begin{align*}
\varphi_{H}(\widetilde{H})=(\lambda-d_{H}(v))\varphi_{H}(\widetilde{H}-v)+(-1)^{k-1}\sum_{e\in E_{\widetilde{H}}(v)}\varphi_{H}(\widetilde{H}-V(e));
\end{align*}
(4) $\frac{\mathrm{d}}{\mathrm{d}\lambda}\varphi_{H}(\widetilde{H})=\sum_{v\in V(\widetilde{H})}\varphi_{H}(\widetilde{H}-v)$.
\end{thm}

\begin{proof}

(1) For any $M\in\mathcal{M}(\widetilde{H})$,
there exists $M_{1}\in\mathcal{M}(\widetilde{H}_{1})$ and $M_{2}\in\mathcal{M}(\widetilde{H}_{2})$ such that $M=M_{1}\cup M_{2}$.
It is easy to check that $\varphi_{H}(\widetilde{H})=\varphi_{H}(\widetilde{H}_{1})\varphi_{H}(\widetilde{H}_{2})$.

(2) For any $M\in\mathcal{M}(\widetilde{H})$,
if $M$ does not contain edge $e$,
then $M$ is a matching of $\widetilde{H}\setminus e$;
if $M$ contain edge $e$,
then $M\setminus\{e\}$ is a matching of $\widetilde{H}-V(e)$.
Thus, we have
\begin{align*}
\varphi_{H}(\widetilde{H})=&\sum_{e\notin M\in\mathcal{M}(\widetilde{H})}(-1)^{(k-1)|M|}\prod_{v\in V(\widetilde{H})\setminus V(M)}(\lambda-d_{H}(v))\\
&+\sum_{e\in M\in\mathcal{M}(\widetilde{H})}(-1)^{(k-1)|M|}\prod_{v\in V(\widetilde{H})\setminus V(M)}(\lambda-d_{H}(v))\\
=&\sum_{M\in\mathcal{M}(\widetilde{H}\setminus e)}(-1)^{(k-1)|M|}\prod_{v\in V(\widetilde{H}\setminus e)\setminus V(M)}(\lambda-d_{H}(v))\\
&+\sum_{M\setminus\{e\}\in\mathcal{M}(\widetilde{H}-V(e))}(-1)^{(k-1)(|M\setminus\{e\}|+1)}\prod_{v\in V(\widetilde{H}-V(e))\setminus V\left(M\setminus\{e\}\right)}(\lambda-d_{H}(v))\\
=&\varphi_{H}(\widetilde{H}\setminus e)+(-1)^{k-1}\varphi_{H}(\widetilde{H}-V(e)).
\end{align*}

(3) Suppose that $I=\{e_{1},\ldots,e_{s}\}$.
It follows from Theorem \ref{dingli3.1} (2) that
\begin{align*}
\varphi_{H}(\widetilde{H})&=\varphi_{H}(\widetilde{H}\setminus e_{1})+(-1)^{k-1}\varphi_{H}(\widetilde{H}-V(e_{1}))\\
&=\varphi_{H}(\widetilde{H}\setminus\{e_{1},e_{2}\})+(-1)^{k-1}\varphi_{H}(\widetilde{H}\setminus e_{1}-V(e_{2}))+(-1)^{k-1}\varphi_{H}(\widetilde{H}-V(e_{1})).
\end{align*}
Since $\widetilde{H}\setminus e_{1}-V(e_{2})=\widetilde{H}-V(e_{2})$,
we have
\begin{align*}
\varphi_{H}(\widetilde{H})=\varphi_{H}(\widetilde{H}\setminus\{e_{1},e_{2}\})+(-1)^{k-1}\varphi_{H}(\widetilde{H}-V(e_{2}))+(-1)^{k-1}\varphi_{H}(\widetilde{H}-V(e_{1})).
\end{align*}
Repeatedly using Theorem \ref{dingli3.1} (2), we get
\begin{align}\label{shi3.2}
\varphi_{H}(\widetilde{H})=\varphi_{H}(\widetilde{H}\setminus I)+(-1)^{k-1}\sum_{e\in I}\varphi_{H}(\widetilde{H}-V(e)).
\end{align}
When $I=E_{\widetilde{H}}(v)$,
the vertex $v$ is an isolated vertex on $H\setminus I$.
By (\ref{shi3.2}) and Theorem \ref{dingli3.1} (1), we thus have that
\begin{align*}
\varphi_{H}(\widetilde{H})=(\lambda-d_{H}(v))\varphi_{H}(\widetilde{H}-v)+(-1)^{k-1}\sum_{e\in E_{\widetilde{H}}(v)}\varphi_{H}(\widetilde{H}-V(e)).
\end{align*}

(4) By (\ref{shi3.1}), we have
\begin{align}\label{shi3.3}\notag
\frac{\mathrm{d}}{\mathrm{d}\lambda}\varphi_{H}(\widetilde{H})&=
\sum_{M\in\mathcal{M(\widetilde{H})}}\sum_{v\in V(\widetilde{H})\setminus V(M)}(-1)^{(k-1)|M|}\prod_{v\neq u\in V(\widetilde{H})\setminus V(M)}(\lambda-d_{H}(u))\\
&=\sum_{M\in\mathcal{M(\widetilde{H})}}\sum_{v\in V(\widetilde{H})\setminus V(M)}(-1)^{(k-1)|M|}\prod_{u\in V(\widetilde{H}-v)\setminus V(M)}(\lambda-d_{H}(u)).
\end{align}
For any $v\in V(\widetilde{H})$,
a matching of $\widetilde{H}$ without $v$ is a matching of $\widetilde{H}-v$.
So $\mathcal{M}(\widetilde{H}-v)$ can be seen as the set of all matchings without $v$ in $\widetilde{H}$.
From (\ref{shi3.3}), we obtain
\begin{align*}
\frac{\mathrm{d}}{\mathrm{d}\lambda}\varphi_{H}(\widetilde{H})&=
\sum_{v\in V(\widetilde{H})}\sum_{M\in \mathcal{M}(\widetilde{H}-v)}(-1)^{(k-1)|M|}\prod_{u\in V(\widetilde{H}-v)\setminus V(M)}(\lambda-d_{H}(u))\\
&=\sum_{v\in V(\widetilde{H})}\varphi_{H}(\widetilde{H}-v).
\end{align*}

\end{proof}

Next, we will give a result about the zero roots of the Laplacian matching polynomial of uniform hypertrees.
For this we need a result about the eigenvalues of principal sub-tensor of Laplacian tensor and the relationship between the eigenvalue of weighted adjacency tensor and the weighted matching polynomial.

For a non-empty $S\subseteq V(H)$,
let $\mathcal{L}_{H}[S]=(l_{i_{1}\cdots i_{k}})$ denote the $k$-order $|S|$-dimensional principal sub-tensor of $\mathcal{L}_{H}$,
where $i_{1},\ldots,i_{k}\in S$.
When $S=V(H)\setminus\{v\}$,
$\mathcal{L}_{H}[S]$ is simply written as $\mathcal{L}_{H}(v)$.
A tensor is called a $\mathcal{Z}$-tensor if all of its off-diagonal entries are non-positive.
Clearly, $\mathcal{L}_{H}[S]$ is a $\mathcal{Z}$-tensor for any non-empty $S\subseteq V(H)$.
Applying some properties of $\mathcal{Z}$-tensor,
we obtain the following result.

\begin{lem}\label{yinli3.2}
Let $H$ be a uniform connected hypergraph.
For any non-empty proper sub-set $S\subset V(H)$,
the real eigenvalues of $\mathcal{L}_{H}[S]$ are all greater than zero.
\end{lem}

\begin{proof}

For any non-empty proper sub-set $S\subset V(H)$,
let $\tau(\mathcal{L}_{H}[S])$ denote the minimum real part of all eigenvalues of $\mathcal{L}_{H}[S]$.
For a non-empty proper sub-set $U\subset V(H)$ satisfying $U\supseteq S$,
it is known that $\tau(\mathcal{L}_{H}[U])\leq\tau(\mathcal{L}_{H}[S])$ \citep[Theorem 3.1]{shen2019some}.
Thus, we have
\begin{align*}
\min_{v\in V(H)}\tau(\mathcal{L}_{H}(v))\leq\tau(\mathcal{L}_{H}[S]).
\end{align*}
By \citep[Proposition 2.4]{gowda2015z},
$\tau(\mathcal{L}_{H}(v))$ is the minimum $\mathrm{H}$-eigenvalue of $\mathcal{L}_{H}(v)$ for any $v\in V(H)$.
It is shown that the minimum $\mathrm{H}$-eigenvalue of $\mathcal{L}_{H}(v)$ is greater than zero for uniform connected hypergraph $H$ in \citep[Lemma 2.1 and Theorem 3.1]{bu2018Inverse}.
Then we have $\tau(\mathcal{L}_{H}(v))>0$.
Thus $0<\min_{v\in V(H)}\tau(\mathcal{L}_{H}(v))\leq\tau(\mathcal{L}_{H}[S])$,
which implies that the real eigenvalues of $\mathcal{L}_{H}[S]$ are all greater than zero.

\end{proof}

For a $k$-uniform hypergraph $H$ and the weighting function $\mathbf{w}:V(H)\cup E(H)\rightarrow \mathbb{C}$,
Wan et al. \cite{wan2022spectra} defined the weighted adjacency tensor $\mathcal{A}_{H,\mathbf{w}}=(a_{i_{1}\ldots i_{k}})$, where
\begin{equation*}
a_{i_{1}\cdots i_{k}}=\begin{cases}
\mathbf{w}(v) &\text{if $i_{1}=\cdots=i_{k}=v\in V(H)$}, \\
\frac{\mathbf{w}(e)}{(k-1)!} &\text{if $\{i_{1},\ldots,i_{k}\}=e\in E(H)$}, \\
0 &\text{otherwise}.
\end{cases}
\end{equation*}
They determined all eigenvalues of the weighted adjacency tensor of uniform hypertrees by means of the weighted matching polynomial.

\begin{lem}\citep[Theorem2]{wan2022spectra}\label{yinli3.3}
Let $T=(V(T),E(T))$ be a $k$-uniform hypertree for $k\geq3$.
Let $\mathbf{w}:V(T)\cup E(T)\rightarrow\mathbb{C}$ be a weighting function on $T$.
Then $\lambda$ is an eigenvalue of $\mathcal{A}_{T,\mathbf{w}}$ if and only if there exists a sub-hypertree $\widetilde{T}$ of $T$ (including isolated vertices) such that
$\lambda$ is a root of the weighted matching polynomial
\begin{align*}
\sum_{M\in\mathcal{M}(\widetilde{T})}(-1)^{|M|}\prod_{e\in M}\mathbf{w}(e)^{k}\prod_{v\in V(\widetilde{T})\setminus V(M)}(\lambda-\mathbf{w}(v)).
\end{align*}
\end{lem}

We are now ready to derive the result as follows.

\begin{thm}\label{dingli3.4}
Let $T$ be a $k$-uniform hypertree.
Then zero is a simple root of the polynomial $\varphi_{T}(T)$.
Moreover,
zero is not a root of the polynomial $\varphi_{T}(\widetilde{T})$ for any non-trivial sub-hypertree $\widetilde{T}$ of $T$.
\end{thm}

\begin{proof}

When $k=2$,
$\varphi_{T}(T)$ is the Laplacian matching polynomial of tree $T$.
It is shown that $\varphi_{T}(T)$ is equal to the Laplacian characteristic polynomial of $T$ in the \citep[Theorem3.3]{mohammadian2020laplacian}.
Since zero is a simple root of the Laplacian characteristic polynomial of $T$,
zero is a simple root of $\varphi_{T}(T)$.
By \citep[Theorem 2.7]{wan2022Onthelocation}, for any non-trivial sub-tree $\widetilde{T}$ of $T$,
it is easy to check that $\varphi_{T}(\widetilde{T})$ is equal to the characteristic polynomial of the Laplacian principal sub-matrix $L_{T}(w)$ of $T$.
Since zero is not a root of the characteristic polynomial of $L_{T}(w)$,
zero is not a root of $\varphi_{T}(\widetilde{T})$.
In the following,
we consider the case $k\geq3$.

Clearly, for any sub-hypertree $\widetilde{T}$ of $T$,
if we choose the weighting function $\mathbf{w}$ on $\widetilde{T}$ such that $\mathbf{w}(v) = d_{T}(v)$ for all $v\in V(\widetilde{T})$ and $\mathbf{w}(e)=-1$ for all $e\in E(\widetilde{T})$,
then $\mathcal{A}_{\widetilde{T},\mathbf{w}}$ is exactly the principal sub-tensor $\mathcal{L}_{T}[V(\widetilde{T})]$ of $\mathcal{L}_{T}$,
and the weighted matching polynomial of $\widetilde{T}$ is exactly $\varphi_{T}(\widetilde{T})$.
It follows from Lemma \ref{yinli3.3} that the roots of $\varphi_{T}(\widetilde{T})$ is the eigenvalues of $\mathcal{L}_{T}[V(\widetilde{T})]$.
When $\widetilde{T}$ is a non-trivial sub-hypertree of $T$, by Lemma \ref{yinli3.2},
we know that zero is not the eigenvalue of $\mathcal{L}_{T}[V(\widetilde{T})]$,
which implies that zero is not a root of the polynomial $\varphi_{T}(\widetilde{T})$.
Since zero is a Laplacian eigenvalue of $T$, by \citep[Corollary4]{wan2022spectra},
there exists a sub-hypertree of $T$ such that zero is the root of the Laplacian matching polynomial of it with respect to $T$.
It is known that zero is not a root of $\varphi_{T}(\widetilde{T})$ for any non-trivial sub-hypertree $\widetilde{T}$ of $T$,
which implies that zero is a root of $\varphi_{T}(T)$.
Next, we prove that zero is a simple root of $\varphi_{T}(T)$.

By Theorem \ref{dingli3.1} (4), we have
\begin{align}\label{shi3.4}
\frac{\mathrm{d}}{\mathrm{d}\lambda}\varphi_{T}(T)=\sum_{v\in V(T)}\varphi_{T}(T-v).
\end{align}
Given a vertex $v\in V(T)$,
we know that $T-v$ is not connected and each connected component is sub-hypertree of $T$.
By Theorem \ref{dingli3.1} (1),
the roots of $\varphi_{T}(T-v)$ are the eigenvalues of $\mathcal{L}_{T}[V(T-v)]$.
By Lemma \ref{yinli3.2}, the real eigenvalues of $\mathcal{L}_{T}[V(T-v)]$ are all greater than zero,
which implies that all real roots of $\varphi_{T}(T-v)$ are greater than zero.
Note that $\varphi_{T}(T-v)$ is a real coefficient polynomial, whose all of imaginary part non-zero complex roots occur in pairs.
So the product of all roots of $\varphi_{T}(T-v)$ is greater than zero.
Let $\lambda_{1}^{(v)},\ldots,\lambda_{|V(T)|-1}^{(v)}$ denote the roots of $\varphi_{T}(T-v)$ for each $v\in V(T)$ and
we have $\lambda_{1}^{(v)}\cdots\lambda_{|V(T)|-1}^{(v)}>0$.
Then the constant term of the polynomial $\sum_{v\in V(T)}\varphi_{T}(T-v)$ is
$(-1)^{|V(T)|-1}\sum_{v\in V(T)}\lambda_{1}^{(v)}\cdots\lambda_{|V(T)|-1}^{(v)}\neq0$,
which implies that zero is not a root of $\sum_{v\in V(T)}\varphi_{T}(T-v)$.
By (\ref{shi3.4}), zero is not a root of $\frac{\mathrm{d}}{\mathrm{d}\lambda}\varphi_{T}(T)$.
Thus, zero is a simple root of $\varphi_{T}(T)$.

\end{proof}

\section{The multiplicity of the zero Laplacian eigenvalue of uniform hypertrees}\label{section4}

In this section, we apply the Laplacian characteristic polynomial and the Laplacian matching polynomial to
give the multiplicity of the zero Laplacian eigenvalue of uniform hypertrees,
which shows that Conjecture \ref{caixiang1.1} is true.

For a $k$-uniform hypertree $T=(V(T),E(T))$ and a vertex $w\in V(T)$,
recall that $F_{v}=F_{v}(x_{i}:i\in V(T))=(\lambda-d_{T}(v))x_{v}^{k-1}+\sum_{e\in E_{T}(v)}\mathbf{x}_{e\setminus\{v\}}$ and $f_{v}=F_{v}|_{x_{w}=1}$ for all $v\in V(T)$.
Let $\mathcal{V}^{T}$ be the affine variety defined by the polynomials $f_{v}$ for all $v\in V(T)\setminus\{w\}$.
By Lemma \ref{yinli2.1},
the Laplacian characteristic polynomial of $T$ is
\begin{align}\label{shi4.1}\notag
\phi(\mathcal{L}_{T})&=\phi(\mathcal{L}_{T}(w))^{k-1}\prod_{\mathbf{p}\in\mathcal{V}^{T}}(\lambda-d_{T}(w)+\sum_{e\in E_{T}(w)}\mathbf{p}_{e\setminus\{w\}})^{m(\mathbf{p})}\\
&=\phi(\mathcal{L}_{T}(w))^{k-1}\prod_{\mathbf{p}\in\mathcal{V}^{T}}f_{w}(\mathbf{p})^{m(\mathbf{p})}.
\end{align}
From Lemma \ref{yinli3.2},
we know that zero is not the eigenvalue of $\mathcal{L}_{T}(w)$.
Hence, in order to determine the multiplicity of the zero Laplacian eigenvalue of $T$,
we only need to consider $\prod_{\mathbf{p}\in\mathcal{V}^{T}}f_{w}(\mathbf{p})^{m(\mathbf{p})}$ in (\ref{shi4.1}).

Let $\mathbf{p}=(p_{i})$ be a point in affine variety $\mathcal{V}^{T}$,
and let $\mathbf{q}=(q_{i})$ be a $|V(T)|$-dimensional vector with components $q_{w}=1$ and $q_{i}=p_{i}$ for all $i\in V(T)\setminus\{w\}$.
Then we have
\begin{align*}
f_{w}(\mathbf{p})=F_{w}(q_{i}:i\in V(T))=F_{w}(\mathbf{q}),
\end{align*}
and $f_{v}(\mathbf{p})=F_{v}(q_{i}:i\in V(T))=F_{v}(\mathbf{q})=0$ for all $v\in V(T)\setminus\{w\}$.
When $\lambda=0$.
If $F_{w}(\mathbf{q})=0$,
then $\mathbf{q}$ is an eigenvector corresponding to the zero Laplacian eigenvalue of $T$.
It is shown that all components of the eigenvector corresponding to the zero Laplacian eigenvalue of a connected uniform hypergraph are non-zero in the \citep[Theorem 4.1 (i)]{hu2014eigenvectors}.
Therefore, the all components of $\mathbf{p}\in\mathcal{V}^{T}$ satisfying $f_{w}(\mathbf{p})=0$ are non-zero when $\lambda=0$.
It implies that the multiplicity of the zero Laplacian eigenvalue of $T$ is only related to the points having all components non-zero in $\mathcal{V}^{T}$.

\begin{lem}\label{yinli4.1}
Let $T$ be a $k$-uniform hypertree and $w$ be a vertex on $T$.
If $\mathbf{p}\in\mathcal{V}^{T}$ have all components non-zero,
then
\begin{align*}
\mathbf{p}_{e\setminus\{w\}}=\frac{(-1)^{k-1}\varphi_{T}(T-V(e))}{\varphi_{T}(T-w)}
\end{align*}
for each $e\in E_{T}(w)$.
\end{lem}

\begin{proof}

We prove the result by the induction on the number of edges of $T$.

When $|E(T)|=1$, we have $\varphi_{T}(T-w)=(\lambda-1)^{k-1}$ and $\varphi_{T}(T-V(e))=1$ for the edge $e\in E_{T}(w)$.
From (\ref{shi2.9}), we know that $\mathbf{p}_{e\setminus\{w\}}=(\frac{-1}{\lambda-1})^{k-1}$,
which implies that
\begin{align*}
\mathbf{p}_{e\setminus\{w\}}=\frac{(-1)^{k-1}\varphi_{T}(T-V(e))}{\varphi_{T}(T-w)}.
\end{align*}
So the assertion holds.

Assuming that the result holds for any $|E(T)|\leq r$, we consider the case $|E(T)|=r+1$.

When $w$ is a cut vertex of $T$,
$T$ has $d_{T}(w)$$(>1)$ branches associated with $w$ and each $e\in E_{T}(w)$ belongs to a distinct branch.
Let $\widetilde{T}_{i}$ be the branch of $T$ with edge $e_{i}\in E_{T}(w)$ for each $i\in[d_{T}(w)]$
and we know that $|E(\widetilde{T}_{i})|\leq r$.
By the induction hypothesis,
for $\mathbf{p}^{(i)}\in\mathcal{V}^{\widetilde{T}_{i}}$ having all components non-zero, we have
\begin{align*}
\mathbf{p}^{(i)}_{e_{i}\setminus\{w\}}=\frac{(-1)^{k-1}\varphi_{\widetilde{T}_{i}}(\widetilde{T}_{i}-V(e_{i}))}{\varphi_{\widetilde{T}_{i}}(\widetilde{T}_{i}-w)}.
\end{align*}
By the definition of the Laplacian matching polynomial,
we have $\varphi_{\widetilde{T}_{i}}(\widetilde{T}_{i}-V(e_{i}))=\varphi_{T}(\widetilde{T}_{i}-V(e_{i}))$ and
$\varphi_{\widetilde{T}_{i}}(\widetilde{T}_{i}-w)=\varphi_{T}(\widetilde{T}_{i}-w)$.
Then
\begin{align}\label{shi4.2}\notag
\mathbf{p}^{(i)}_{e_{i}\setminus\{w\}}&=\frac{(-1)^{k-1}\varphi_{T}(\widetilde{T}_{i}-V(e_{i}))}{\varphi_{T}(\widetilde{T}_{i}-w)}\\
&=\frac{(-1)^{k-1}\varphi_{T}(\widetilde{T}_{i}-V(e_{i}))\prod_{\substack{j\in[d_{T}(w)]\\j\neq i}}\varphi_{T}(\widetilde{T}_{j}-w)}{\prod_{j\in[d_{T}(w)]}\varphi_{T}(\widetilde{T}_{j}-w)}.
\end{align}
Note that $T-w$ is the disjoint union of $\widetilde{T}_{i}-w$ for all $i\in[d_{T}(w)]$,
and $T-V(e_{j})$ is the disjoint union of $\widetilde{T}_{j}-V(e_{j})$ and $\widetilde{T}_{i}-w$ for all $i\neq j\in[d_{T}(w)]$.
It follows from Theorem \ref{dingli3.1} (1) that
\begin{align*}
\prod_{j\in[d_{T}(w)]}\varphi_{T}(\widetilde{T}_{j}-w)=\varphi_{T}(T-w),
\end{align*}
and
\begin{align*}
\varphi_{T}(\widetilde{T}_{i}-V(e_{i}))\prod_{\substack{j\in[d_{T}(w)]\\j\neq i}}\varphi_{T}(\widetilde{T}_{j}-w)=\varphi_{T}(T-V(e_{i})).
\end{align*}
By Theorem \ref{dingli2.4} and (\ref{shi4.2}),
for $\mathbf{p}\in\mathcal{V}^{T}$ having all components non-zero, we get
\begin{align*}
\mathbf{p}_{e_{i}\setminus\{w\}}=\mathbf{p}^{(i)}_{e_{i}\setminus\{w\}}=\frac{(-1)^{k-1}\varphi_{T}(T-V(e_{i}))}{\varphi_{T}(T-w)}.
\end{align*}

When $w$ is not a cut vertex of $T$,
the degree of $w$ is clearly one.
Let the edge $\widehat{e}=\{v_{1},\ldots,v_{k-1},w\}$.
Then $T\setminus\widehat{e}$ has $k$ connected components and we use $\widehat{T}_{t}$ to denote the connected
component containing $v_{t}$ for each $t\in[k]$.

For all $v\in V(T)$,
recall that $F_{v}=F_{v}(x_{i}:i\in V(T))=(\lambda-d_{T}(v))x_{v}^{k-1}+\sum_{e\in E_{T}(v)}\mathbf{x}_{e\setminus\{v\}}$ and $f_{v}=F_{v}|_{x_{w}=1}$.
For all $t\in[k-1]$ and any $v\in V(\widehat{T}_{t})\setminus\{v_{t}\}$,
note that $f_{v}=f_{v}(x_{i}:i\in V(\widehat{T}_{t}))$ is a homogeneous polynomial.
Since $\mathbf{p}=(p_{i})\in\mathcal{V}^{T}$ have all components non-zero,
we get
\begin{align}\label{shi4.3}
f_{v}(\mathbf{p})=f_{v}(p_{i}:i\in V(\widehat{T}_{t}))=f_{v}\left(\frac{p_{i}}{p_{v_{t}}}:i\in V(\widehat{T}_{t})\right)=0.
\end{align}

Fix $t\in[k-1]$,
we consider the sub-hypertree $\widehat{T}_{t}$.
For all $v\in V(\widehat{T}_{t})\setminus\{v_{t}\}$,
let $\widehat{F}_{v}=\widehat{F}_{v}(x_{i}:i\in V(\widehat{T}_{t}))=(\lambda-d_{\widehat{T}_{t}}(v))x_{v}^{k-1}+\sum_{e\in E_{\widehat{T}_{t}}(v)}\mathbf{x}_{e\setminus\{v\}}$ and $\widehat{f}_{v}=\widehat{F}_{v}|_{x_{v_{t}}=1}$.
It is easy to check that $\widehat{F}_{v}=f_{v}$.
Let $q_{i}=\frac{p_{i}}{p_{v_{t}}}$ for all $i\in V(\widehat{T}_{t})$ and note that $q_{v_{t}}=1$.
By (\ref{shi4.3}), we have
\begin{align}\label{shi4.4}
\widehat{f}_{v}(q_{i}:i\in V(\widehat{T}_{t})\setminus\{v_{t}\})=\widehat{F}_{v}(q_{i}:i\in V(\widehat{T}_{t}))=f_{v}(q_{i}:i\in V(\widehat{T}_{t}))=0
\end{align}
for all $v\in V(\widehat{T}_{t})\setminus\{v_{t}\}$.
Let the vector $\mathbf{q}=(q_{i})$ for $i\in V(\widehat{T}_{t})\setminus\{v_{t}\}$.
Then $\mathbf{q}$ is a point in the affine variety $\mathcal{V}^{\widehat{T}_{t}}$ defined by the polynomials $\widehat{f}_{v}$ for all $v\in V(\widehat{T}_{t})\setminus\{v_{t}\}$,
and the all components of $\mathbf{q}$ are non-zero.
By the induction hypothesis,
for each $e\in E_{\widehat{T}_{t}}(v_{t})$, we have
\begin{align*}
\mathbf{q}_{e\setminus\{v_{t}\}}=\frac{(-1)^{k-1}\varphi_{\widehat{T}_{t}}(\widehat{T}_{t}-V(e))}{\varphi_{\widehat{T}_{t}}(\widehat{T}_{t}-v_{t})}.
\end{align*}
By the definition of the Laplacian matching polynomial,
we have $\varphi_{\widehat{T}_{t}}(\widehat{T}_{t}-V(e))=\varphi_{T}(\widehat{T}_{t}-V(e))$ and
$\varphi_{\widehat{T}_{t}}(\widehat{T}_{t}-v_{t})=\varphi_{T}(\widehat{T}_{t}-v_{t})$.
Then
\begin{align*}
\mathbf{q}_{e\setminus\{v_{t}\}}=\frac{(-1)^{k-1}\varphi_{T}(\widehat{T}_{t}-V(e))}{\varphi_{T}(\widehat{T}_{t}-v_{t})}=\frac{\mathbf{p}_{e\setminus\{v_{t}\}}}{p_{v_{t}}^{k-1}}.
\end{align*}
Thus, for $\mathbf{p}\in\mathcal{V}^{T}$ having all components non-zero and each $e\in E_{\widehat{T}_{t}}(v_{t})$, we get
\begin{align}\label{shi4.5}
\mathbf{p}_{e\setminus\{v_{t}\}}=\frac{(-1)^{k-1}\varphi_{T}(\widehat{T}_{t}-V(e))}{\varphi_{T}(\widehat{T}_{t}-v_{t})}p_{v_{t}}^{k-1}.
\end{align}

For each $t\in[k-1]$, recall that
\begin{align*}
f_{v_{t}}(\mathbf{p})=(\lambda-d_{T}(v_{t}))p_{v_{t}}^{k-1}+\mathbf{p}_{\widehat{e}\setminus\{v_{t},w\}}+\sum_{e\in E_{\widehat{T}_{t}}(v_{t})}\mathbf{p}_{e\setminus\{v_{t}\}}=0.
\end{align*}
By (\ref{shi4.5}) and Theorem \ref{dingli3.1} (3), we have
\begin{align*}
\mathbf{p}_{\widehat{e}\setminus\{v_{t},w\}}&=-\left(\lambda-d_{T}(v_{t})+\sum_{e\in E_{\widehat{T}_{t}}(v_{t})}\frac{(-1)^{k-1}\varphi_{T}(\widehat{T}_{t}-V(e))}{\varphi_{T}(\widehat{T}_{t}-v_{t})}\right)p_{v_{t}}^{k-1}\\
&=-\frac{\varphi_{T}(\widehat{T}_{t})}{\varphi_{T}(\widehat{T}_{t}-v_{t})}p_{v_{t}}^{k-1}.
\end{align*}
Combining these equations for all $t\in[k-1]$, we get
\begin{align*}
\prod_{t=1}^{k-1}\mathbf{p}_{\widehat{e}\setminus\{v_{t},w\}}&=(-1)^{k-1}\prod_{t=1}^{k-1}\frac{\varphi_{T}(\widehat{T}_{t})}{\varphi_{T}(\widehat{T}_{t}-v_{t})}p_{v_{t}}^{k-1}.
\end{align*}
Since $\prod_{t=1}^{k-1}\mathbf{p}_{\widehat{e}\setminus\{v_{t},w\}}=\prod_{t=1}^{k-1}p_{v_{t}}^{k-2}$,
we have
\begin{align*}
\mathbf{p}_{\widehat{e}\setminus\{w\}}=\frac{\prod_{t=1}^{k-1}p_{v_{t}}^{k-1}}{\prod_{t=1}^{k-1}\mathbf{p}_{\widehat{e}\setminus\{v_{t},w\}}}=(-1)^{k-1}\prod_{t=1}^{k-1}\frac{\varphi_{T}(\widehat{T}_{t}-v_{t})}{\varphi_{T}(\widehat{T}_{t})}.
\end{align*}
Note that for all $t\in[k-1]$, the disjoint union of $\widehat{T}_{t}-v_{t}$ is $T-V(\widehat{e})$ and the disjoint union of $\widehat{T}_{t}$ is $T-w$.
It follows from Theorem \ref{dingli3.1} (1) that
\begin{align*}
\mathbf{p}_{\widehat{e}\setminus\{w\}}=\frac{(-1)^{k-1}\varphi_{T}(T-V(\widehat{e}))}{\varphi_{T}(T-w)}.
\end{align*}

\end{proof}

For the point $\mathbf{p}\in\mathcal{V}^{T}$, we have
$f_{w}(\mathbf{p})=\lambda-d_{T}(w)+\sum_{e\in E_{T}(w)}\mathbf{p}_{e\setminus\{w\}}$.
If $\mathbf{p}$ have all components non-zero, by Lemma \ref{yinli4.1} and Theorem \ref{dingli3.1} (3), we get
\begin{align}\label{shi4.6}\notag
f_{w}(\mathbf{p})&=\lambda-d_{T}(w)+\sum_{e\in E_{T}(w)}\frac{(-1)^{k-1}\varphi_{T}(T-V(e))}{\varphi_{T}(T-w)}\\
&=\frac{\varphi_{T}(T)}{\varphi_{T}(T-w)}.
\end{align}
Note that $T-w$ is not connected and each connected component is a non-trivial sub-hypertree of $T$.
From Theorem \ref{dingli3.1} (1) and Theorem \ref{dingli3.4},
we know that zero is not the root of $\varphi_{T}(T-w)$ and is a simple root of $\varphi_{T}(T)$.
Let $n_{0}(T)$ denote the multiplicity of the zero Laplacian eigenvalue of $T$.
Since $n_{0}(T)$ is only related to $\mathbf{p}$ having all components non-zero in $\mathcal{V}^{T}$,
combining (\ref{shi4.1}) with (\ref{shi4.6}), we have
\begin{align}\label{shi4.7}
n_{0}(T)=\sum_{\substack{\mathbf{p}\in\mathcal{V}^{T}\\\forall p_{i}\neq0}}m(\mathbf{p}),
\end{align}
where $m(\mathbf{p})$ is the multiplicity of $\mathbf{p}=(p_{i})$ in $\mathcal{V}^{T}$.

We are now ready to determine the multiplicity of the zero Laplacian eigenvalue of $T$.

\begin{thm}\label{dingli4.2}
Let $T=(V(T),E(T))$ be a $k$-uniform hypertree.
Then the multiplicity of the zero Laplacian eigenvalue of $T$ is $k^{|E(T)|(k-2)}$.
\end{thm}

\begin{proof}

We prove the result by the induction on the number of edges of $T$.

When $|E(T)|=1$.
It is shown that the multiplicity of the zero Laplacian eigenvalue of $T$ is $k^{k-2}$ in the \cite[Theorem 4.9]{zheng2023zero}.
So the assertion holds.

Assuming that the result holds when $|E(T)|=r$, we consider the case $|E(T)|=r+1$.

Let $w$ be a non-pendent vertex on a pendant edge of $T$,
and $\widetilde{T}$ denote the $k$-uniform hypertree obtained by removing this pendant edge and pendent vertices on it from $T$.
By Corollary \ref{tuilun2.5},
the Laplacian characteristic polynomial of $T$ is
\begin{align}\label{shi4.8}\notag
\phi(\mathcal{L}_{T})=&(\lambda-1)^{(k-1)^{(r+1)(k-1)+1}}\phi(\mathcal{L}_{\widetilde{T}}(w))^{(k-1)^{k}}
\prod_{\mathbf{p}\in\mathcal{V}^{\widetilde{T}}}(\lambda-d_{T}(w)+\sum_{e\in E_{\widetilde{T}}(w)}\mathbf{p}_{e\setminus\{w\}})^{m(\mathbf{p})K_{1}}\\
&\times\prod_{\mathbf{p}\in\mathcal{V}^{\widetilde{T}}}(\lambda-d_{T}(w)+(\frac{-1}{\lambda-1})^{k-1}+\sum_{e\in E_{\widetilde{T}}(w)}\mathbf{p}_{e\setminus\{w\}})^{m(\mathbf{p})K_{2}},
\end{align}
where $K_{1}=(k-1)^{k-1}-k^{k-2}$ and $K_{2}=k^{k-2}$.

Clearly, $w$ is a cut vertex on $T$.
Suppose that the branches of $T$ associated with $w$ are $\widetilde{T}$ and a one-edge hypergraph, denoted by $T'$.
By (\ref{shi2.3}), we know that $\mathcal{V}^{T}=\mathcal{V}^{\widetilde{T}}\times\mathcal{V}^{T'}$.
Then we have $\mathbf{r}=\left(\begin{matrix} \mathbf{p} \\ \mathbf{q} \end{matrix}\right)$ for any $\mathbf{r}\in\mathcal{V}^{T}$,
where $\mathbf{p}\in\mathcal{V}^{\widetilde{T}}$, $\mathbf{q}\in\mathcal{V}^{T'}$.
It is known from (\ref{shi4.7}) that the multiplicity of the zero Laplacian eigenvalue of $T$ is only related to $\mathbf{r}\in\mathcal{V}^{T}$
having all components non-zero.
By (\ref{shi2.10}),
it implies that we only need to consider
\begin{align}\label{shi4.9}
\prod_{\mathbf{p}\in\mathcal{V}^{\widetilde{T}}}(\lambda-d_{T}(w)+(\frac{-1}{\lambda-1})^{k-1}+\sum_{e\in E_{\widetilde{T}}(w)}\mathbf{p}_{e\setminus\{w\}})^{m(\mathbf{p})K_{2}}
\end{align}
in (\ref{shi4.8}) and $\mathbf{p}$ have all components non-zero in $\mathcal{V}^{\widetilde{T}}$.

By Lemma \ref{yinli4.1},
for $\mathbf{p}\in\mathcal{V}^{\widetilde{T}}$ having all components non-zero, we have
\begin{align*}
&\lambda-d_{T}(w)+(\frac{-1}{\lambda-1})^{k-1}+\sum_{e\in E_{\widetilde{T}}(w)}\mathbf{p}_{e\setminus\{w\}}\\
=&\lambda-d_{T}(w)+(\frac{-1}{\lambda-1})^{k-1}+\sum_{e\in E_{\widetilde{T}}(w)}\frac{(-1)^{k-1}\varphi_{\widetilde{T}}(\widetilde{T}-V(e))}{\varphi_{\widetilde{T}}(\widetilde{T}-w)}.
\end{align*}
By the definition of the Laplacian matching polynomial,
we know that$\varphi_{\widetilde{T}}(\widetilde{T}-w)=\varphi_{T}(\widetilde{T}-w)$ and
$\varphi_{\widetilde{T}}(\widetilde{T}-V(e))=\varphi_{T}(\widetilde{T}-V(e))$ for each $e\in E_{\widetilde{T}}(w)$.
It follows from Theorem \ref{dingli3.1} (3) that
\begin{align}\label{shi4.10}\notag
&\lambda-d_{T}(w)+(\frac{-1}{\lambda-1})^{k-1}+\sum_{e\in E_{\widetilde{T}}(w)}\mathbf{p}_{e\setminus\{w\}}\\
=&\frac{(\lambda-1)^{k-1}\varphi_{T}(\widetilde{T})+(-1)^{k-1}\varphi_{T}(\widetilde{T}-w)}{(\lambda-1)^{k-1}\varphi_{T}(\widetilde{T}-w)}.
\end{align}
Let pendant edge $\widetilde{e}=\{v_{1},\ldots,v_{k-1},w\}$,
where $v_{1},\ldots,v_{k-1}$ are the pendent vertices.
Note that the Laplacian matching polynomial of $v_{i}$ with respect to $T$ is $\lambda-1$ for each $i\in[k-1]$.
Since the disjoint union of $\widetilde{T}-w$ and $v_{i}$ for all $i\in[k-1]$ is $T-w$,
by Theorem \ref{dingli3.1} (1), we have
\begin{align*}
(\lambda-1)^{k-1}\varphi_{T}(\widetilde{T}-w)=\varphi_{T}(T-w).
\end{align*}
Since $(\lambda-1)^{k-1}\varphi_{T}(\widetilde{T})+(-1)^{k-1}\varphi_{T}(\widetilde{T}-w)=(\lambda-d_{T}(v_{i}))\varphi_{T}(T-v_{i})+(-1)^{k-1}\varphi_{T}(T-V(\widetilde{e}))$ for any $i\in[k-1]$,
by Theorem \ref{dingli3.1} (3), we have
\begin{align*}
(\lambda-1)^{k-1}\varphi_{T}(\widetilde{T})+(-1)^{k-1}\varphi_{T}(\widetilde{T}-w)=\varphi_{T}(T).
\end{align*}
From (\ref{shi4.10}),  for $\mathbf{p}\in\mathcal{V}^{\widetilde{T}}$ having all components non-zero, we obtain
\begin{align*}
&\lambda-d_{T}(w)+(\frac{-1}{\lambda-1})^{k-1}+\sum_{e\in E_{\widetilde{T}}(w)}\mathbf{p}_{e\setminus\{w\}}\\
&=\frac{\varphi_{T}(T)}{\varphi_{T}(T-w)}.
\end{align*}
Note that $T-w$ is not connected and each connected component is a non-trivial sub-hypertree of $T$.
It is known from Theorem \ref{dingli3.1} (1) and Theorem \ref{dingli3.4}
that zero is not the root of $\varphi_{T}(T-w)$ and is a simple root of $\varphi_{T}(T)$.
By (\ref{shi4.9}), we get
\begin{align*}
n_{0}(T)=k^{k-2}\sum_{\substack{\mathbf{p}\in\mathcal{V}^{\widetilde{T}}\\\forall p_{i}\neq0}}m(\mathbf{p}).
\end{align*}
It follows from (\ref{shi4.7}) that
$\sum_{\substack{\mathbf{p}\in\mathcal{V}^{\widetilde{T}}\\\forall p_{i}\neq0}}m(\mathbf{p})=n_{0}(\widetilde{T})$.
By the induction hypothesis,
we have $n_{0}(\widetilde{T})=k^{r(k-2)}$.
Thus, $n_{0}(T)=k^{k-2}n_{0}(\widetilde{T})=k^{(r+1)(k-2)}$.

\end{proof}

\section*{References}
\bibliographystyle{plain}
\bibliography{atbib}
\end{spacing}
\end{document}